\documentclass{birkjour}
\usepackage{comment}

\usepackage{subfigure}
 \newtheorem{theorem}{Theorem}[section]
 \newtheorem{corollary}[theorem]{Corollary}
 \newtheorem{lemma}[theorem]{Lemma}
 \newtheorem{proposition}[theorem]{Proposition}

 \numberwithin{equation}{section}
 
 \newcommand{\R}{\mathbb{R}}

 \newcommand{\fsep}{\hspace*{\fill}}

\begin{document}

%
%
%
%
%
%
%
%
%

\title[Polygons with Parallel Opposite Sides]
  {\centerline{Polygons with Parallel Opposite Sides}}

\author[M.Craizer]{Marcos Craizer}

\address{%
Departamento de Matem\'{a}tica- PUC-Rio\br
Rio de Janeiro\br
Brazil}

\email{craizer@puc-rio.br}

\thanks{The first and second authors want to thank CNPq for financial support during the preparation of this manuscript.}
\author[R.C.Teixeira]{Ralph C.Teixeira}
\address{Departamento de Matem\'{a}tica Aplicada- UFF \br
Niter\'oi\br
Brazil}
\email{ralph@mat.uff.br}

\author[M.A.H.B. da Silva]{Moacyr A.H.B. da Silva}
\address{Centro de Matem\'{a}tica Aplicada- FGV \br
Rio de Janeiro\br
Brazil}
\email{moacyr@fgv.br}



\date{August 14, 2012}

\begin{abstract}
In this paper we consider convex planar polygons with parallel opposite sides. This type of polygons can be regarded as discretizations
of  closed convex planar curves by taking tangent lines at samples with pairwise parallel tangents. For such polygons,
we define discrete versions of the area evolute, central symmetry set, equidistants
and area parallels and show that they behave quite similarly to their smooth counterparts. 
\end{abstract}

\maketitle


\section{Introduction}

In the emerging field of discrete differential geometry, one tries to define discrete counterparts of concepts from 
differential geometry that preserve most of the properties of the smooth case. The idea is to discretize not just the equations, but the whole theory (\cite{Bobenko08}).
Besides being important in computer applications, 
good discrete models may shed light on some aspects of the theory that remain hidden in the smooth context. It is also a common belief
that a good discrete counterpart of a differential geometry concept leads to efficient numerical algorithms for computing it.

There are several affine symmetry sets and evolutes associated with closed planar curves. 
In \cite{Craizer12}, we have considered convex equal area polygons, which can be regarded as discretizations of closed convex planar curves
by uniform sampling with respect to affine arc-length. For this type of polygons, discrete notions of affine normal and curvature, affine evolute,
parallels and affine distance symmetry set are quite natural.

In this paper we consider convex planar polygons with parallel opposite sides. This type of polygons can be regarded as discretizations
of  closed convex planar curves by taking tangent lines at samples with pairwise parallel tangents. For such polygons,
we shall define discrete notions of the area evolute, central symmetry set, equidistants
and area parallels and show that they behave quite similarly to the smooth case.

Given a closed convex planar smooth curve, the {\it area evolute} is defined as the locus 
of midpoints of chords with parallel tangents, while the {\it central symmetry set} is the envelope of chords with parallel tangents (\cite{Giblin08},\cite{Janeczko96}). 
It is well-known that the Area Evolute and Central Symmetry Set have the following properties: 
\begin{itemize}
\item For a point $x$ inside the curve, consider chords with $x$ as midpoint and denote by $N(x)$ the number of such chords.
Then $N(x)$ changes by $2$ when you traverse the Area Evolute (\cite{Giblin08}). 
\item Both Area Evolute and Central Symmetry Set have an odd number of cusps, at least three, and the number of cusps of the Central Symmetry Set is at least the number of cusps of the Area Evolute (\cite{Holtom99},\cite{Rios11}). 
\item Both Area Evolute and Central Symmetry Set reduce to a point if and only if  the curve is symmetric with respect to a point (\cite{Meyer91}).
\end{itemize}

For a convex polygon with parallel opposite sides, we call great diagonals the lines connecting opposite vertices and define the {\it area evolute}  and {\it central symmetry set} as follows: 
the Central Symmetry Set is the polygonal line whose vertices are the intersection of great diagonals, while 
the Area Evolute is the polygonal line connecting midpoints of the same great diagonals.  We also define the notion of cusps for these polygonal lines and 
then prove that all the above mentioned properties remain true.

An equidistant of a closed convex planar smooth curve is the locus of points that belong to a parallel tangents chord with a fixed ratio with respect
to its extremities.  The following property is well-known:
\begin{itemize}
\item  The Central Symmetry Set is the locus of cusps of equidistants (\cite{Giblin08}).  
\end{itemize}
We are not aware of any work concerning the set of self-intersections of equidistants.

We define the equidistants  of a convex polygon with parallel opposite sides as the polygon whose vertices are along the great diagonals with a fixed ratio with respect to the original vertices.
We shall verify that, as in the smooth case, cusps of the equidistants coincide with the Central Symmetry Set. 
In this discrete setting, the locus of self-intersections of the equidistants is also an interesting set to consider. We call it the {\it equidistant symmetry set} 
and prove that the branches of the Equidistant Symmetry Set have cusps at the Central Symmetry Set and endpoints at cusps of the Central Symmetry Set or cusps of the Area Evolute.

For a smooth convex closed planar curve, an area parallel is the envelope of the chords that cut off a fixed area of the interior of the curve (\cite{Craizer08},\cite{Niethammer04}). 
The following property is well-known:
\begin{itemize}
\item The locus of cusps of the area parallels is exactly the Area Evolute (\cite{Holtom01}, sec.2.7). 
\end{itemize}

For a convex polygon with parallel opposite sides, the definition of area parallels can be carried out without changes and we obtain piecewise smooth curves made of arcs of hyperbolas. Nevertheless, it seems easier to consider the polygons whose sides connect the endpoints of these arcs of hyperbolas, that we call 
{\it rectified area parallels}. We show that, as in the smooth case, cusps of the rectified area parallels belong to the area evolute.

Given a convex polygon ${\mathcal P}$ with parallel opposite sides, we show the existence of a convex polygon ${\mathcal Q}$ with parallel opposite sides whose
Central Symmetry Set is exactly the Area Evolute of ${\mathcal P}$ and whose sides are parallel to the great diagonals of ${\mathcal P}$. It is uniquely defined 
up to equidistants and we call it the {\it parallel-diagonal transform} of ${\mathcal P}$. 

In the smooth case, the locus of self-intersections of the area parallels is called {\it affine area symmetry set} (\cite{Holtom01}, sec.2.7). 
For a convex polygon ${\mathcal P}$ with parallel opposite sides, the locus of self-intersections of rectified area parallels will be called {\it rectified area symmetry set}.
The Rectified Area Symmetry Set is difficult to understand in general, but we can describe it when the original polygon satisfies an {\it almost symmetry} hypothesis. A $1$-diagonal of ${\mathcal P}$ is a diagonal connecting vertices such that one is adjacent to the opposite of the other. The almost symmetry  hypothesis states that for some $\mu_0$, the equidistant ${\mathcal Q}_{\mu_0}$ of ${\mathcal Q}$ contains the Area Evolute of ${\mathcal P}$ in its interior and the midpoints of the $1$-diagonals are outside it. 
Under this hypothesis, the equidistants ${\mathcal Q}_{\mu}$ for $\mu\geq\mu_0$ coincide with the area parallels of ${\mathcal P}$ and thus the Equidistant Symmetry Set of ${\mathcal Q}$ coincides with the Rectified Area Symmetry Set
of ${\mathcal P}$. We conclude that the branches of the Rectified Area Symmetry Set of ${\mathcal P}$ have cusps
at the Area Evolute of ${\mathcal P}$ and endpoints at the cusps of the Area Evolute of ${\mathcal P}$ or else at cusps of the area evolute ${\mathcal N}$ of 
${\mathcal Q}$. Observe also that, under the almost symmetry hypothesis, ${\mathcal N}$ is exactly
the rectified area parallel of level half of the total area of ${\mathcal P}$.

The paper is organized as follows:  In section 2 we define the Area Evolute and Central Symmetry Set for convex polygons with parallel opposite sides 
and prove their basic properties. 
In section 3 we study the equidistants and the corresponding symmetry set. 
In section 4 we define and prove the existence of the parallel-diagonal transform. In section 5 we discuss the properties of the rectified area parallels.  

 We have used the free software GeoGebra (\cite{GeoGebra}) for all figures and many experiments during the preparation of the paper. Applets of some of these experiments are available at \cite{Applets}, and we 
 refer to them in the text simply by Applets.
 We would like to thank the GeoGebra team for this excellent mathematical tool.

\section{Basic notions}

A closed planar polygon ${\mathcal P}$ is called convex if it bounds a convex region and has not parallel adjacent sides. 
Let $\{P_1,...,P_n,P_{n+1},...,P_{2n}\}$ denote the vertices of a convex planar $2n$-gon ${\mathcal P}$. 
The polygon has parallel opposite sides  if
$$
\left( P_{i+n+1}-P_{i+n} \right) \parallel  \left( P_{i+1}-P_i \right),
$$
for any $1\leq i\leq n$. Throughout the paper, indices will be taken modulo $2n$. 
We shall assume that the convex parallel opposite sides (CPOS) polygons ${\mathcal P}$ are positively oriented, i.e., 
$$
\left[     P_{i+1}-P_i,     P_{j+1}-P_j    \right]>0,
$$
for any $1\leq i <j \leq n$.

\subsection{ Area Evolute and Central Symmetry Set}

Denote by $d_i$ the great diagonal $P_iP_{i+n}$ and by $D(i+\tfrac{1}{2})$ the point of intersection of  $d_{i}$ and $d_{i+1}$. 
The {\it central symmetry set}  of the polygon ${\mathcal P}$ is the polygon whose vertices are $D(i+\frac{1}{2})$, $1\leq i\leq n$.

Denote by $M_{i}=\frac{1}{2}\left( P_i+P_{i+n} \right)$ the midpoints of the segments $P_iP_{i+n}$ and by $m(i+\tfrac{1}{2})$ 
the lines through $M_i$ parallel to $P_iP_{i+1}$, which we shall call mid-parallel lines. 
The polygon whose vertices are $M_i$, $1\leq i\leq n$, is called the {\it area evolute} of the polygon (see figure \ref{Basics} and Applets).

\begin{figure}[htb]
 \centering
 \includegraphics[width=1.00\linewidth]{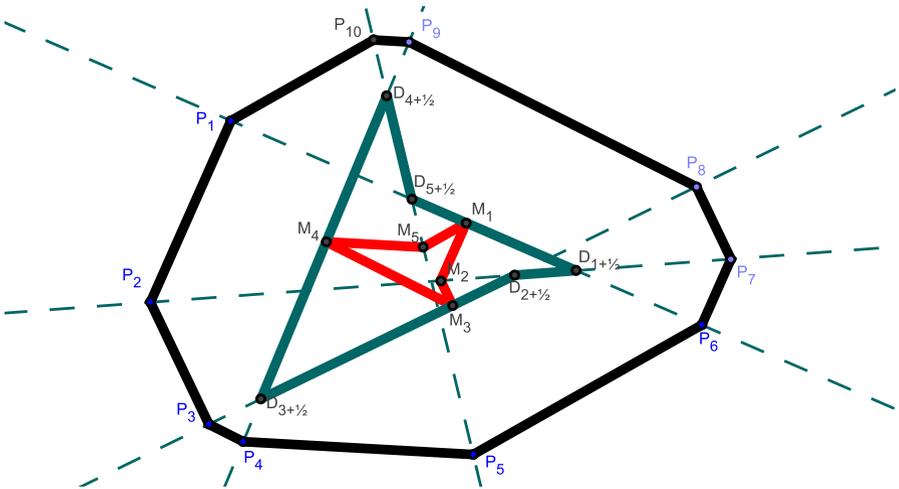}
 \caption{Great diagonals dashed. The area evolute has vertices $M_i$, while the central symmetry set has vertices $D_{i+1/2}$.}
\label{Basics}
\end{figure}

A CPOS polygon is called {\it symmetric} with respect to a point $O$ if $P_{i+n}-O=O-P_i$, for any $1\leq i\leq n$.

\begin{proposition}
The central symmetry set of a CPOS polygon reduces to a point if and only if the polygon is symmetric. Similarly, the area evolute of a CPOS polygon 
reduces to a point if and only if the polygon is symmetric.
\end{proposition}
\begin{proof}
It is clear that if the CPOS polygon is symmetric with respect to a point $O$, then the Area Evolute and the Central Symmetry Set reduce to this point. It is also clear that if
the Area Evolute of a polygon reduces to a point $O$, then the polygon is symmetric with respect to $O$.

Assume now that the Central Symmetry Set ${\mathcal P}$ reduces to a point $O$. Then necessarily 
$$
P_{i+n+1}-P_{i+n}=-\alpha\left(P_{i+1}-P_i\right),
$$
for some $\alpha>0$. But such a polygon can close only if $\alpha=1$, and thus ${\mathcal P}$ is symmetric.
\end{proof}

\subsection{The mid-point property}\label{sec:mid-point}

We shall denote by $e(i+\tfrac{1}{2})$ the edge whose endpoints are $P_i$ and $P_{i+1}$. A pair of edges $e(i+\tfrac{1}{2})$ and $e(j+\tfrac{1}{2})$ determines a parallelogram 
$M(i+\frac{1}{2},j+\frac{1}{2})$ consisting of the mid-points of pairs $(y_1,y_2)$, $y_1\in e(i+\tfrac{1}{2})$ and $y_2\in e(j+\tfrac{1}{2})$.

Denote by $N(x)$ the number of chords with midpoint $x$, identifying those chords whose endpoints belong to the same sides of the polygon. 

\begin{proposition}\label{prop:midpoints}
$N(x)$ is locally constant except at the area evolute. 
When one traverses a segment of the area evolute, $N(x)$ increases or decreases by $2$, depending on the orientation. 
\end{proposition}
\begin{proof}

Write $x=\frac{1}{2}(x_1+x_2)$, where $x_1\in e(i+\tfrac{1}{2})$ and $x_2\in e(j+\tfrac{1}{2})$. We shall assume that $x_1$ and $x_2$ are not both endpoints 
of $e(i+\tfrac{1}{2})$ and $e(j+\tfrac{1}{2})$, or equivalently, $x$ is not a vertex of $M(i+\frac{1}{2},j+\frac{1}{2})$. 

We must consider the following cases:
\begin{enumerate}

\item If $x_1$ and $x_2$ are both in the interior of $e(i+\tfrac{1}{2})$ and $e(j+\tfrac{1}{2})$,\linebreak $j\neq i+n$, or equivalently, $x$ is in interior of $M(i+\frac{1}{2},j+\frac{1}{2})$,  then for $y$ in a neighborhood of $x$,
we can find unique $y_1\in e(i+\tfrac{1}{2})$ and $y_2\in e(j+\tfrac{1}{2})$ such that $y=\frac{1}{2}(y_1+y_2)$.  

\item If $x_1=P_i$ and $j\neq i+n$, $j\neq i+n-1$, then again there exists a neighborhood $U$ of $x$ such that for $y\in U$  we can find unique 
$y_1\in e(i-\tfrac{1}{2})\cup e(i+\tfrac{1}{2})$ and $y_2\in e(j+\tfrac{1}{2})$ such that $y=\frac{1}{2}(y_1+y_2)$. 
In fact, for $y\in U\cap M(i-\frac{1}{2},j+\frac{1}{2})$ we choose $y_1\in e(i-\tfrac{1}{2})$ while for  for $y\in U\cap M(i+\frac{1}{2},j+\frac{1}{2})$ we choose $y_1\in e(i+\tfrac{1}{2})$.

\item Suppose that $x_1=P_i$ and $x_2\in e(i+n-\tfrac{1}{2})$ with $|e(i-\tfrac{1}{2})|>|e(i+n-\tfrac{1}{2})|$ (see figure \ref{MidPoint1}). In this case $x$ is also the midpoint of $x_2'=P_{i+n-1}$ and some
$x_1'\in e(i-\tfrac{1}{2})$. Thus, for $y$ in a neighborhood of $x$, we distinguish between $y\in M(i+\frac{1}{2},i+n-\frac{1}{2})$ and $y\in M(i-\frac{1}{2},i+n-\frac{3}{2})$. In the first case, 
$y$ is the midpoint of $y_1\in e(i+\frac{1}{2})$ and $y_2\in e(i+n-\frac{1}{2})$. In the second case, $y$ is the midpoint of $y_1\in e(i-\frac{1}{2})$ and $y_2\in e(i+n-\frac{3}{2})$. 

\begin{figure}[htb]
 \centering
 \includegraphics[width=1.00\linewidth]{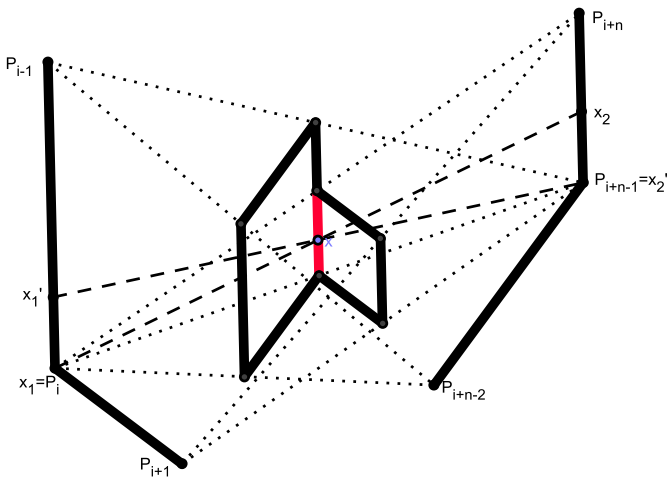}
 \caption{Case 3. $x$ is in the segment intersection of both parallelograms.}
\label{MidPoint1}
\end{figure}

\item It remains to analyze the case in which $x$ is on the area evolute. Suppose that $x_2=P_{i+n}$ and $x_1\in e(i-\tfrac{1}{2})$ with $|e(i-\tfrac{1}{2})|>|e(i+n-\tfrac{1}{2})|$ (see figure \ref{MidPoint2}). In this case $x$ is also the midpoint of $x_2'=P_{i+n-1}$ and some $x_1'\in e(i-\tfrac{1}{2})$ and the parallelograms with $x$ in the boundary are $M(i-\frac{1}{2},i+n-\frac{3}{2})$ and 
$M(i-\frac{1}{2},i+n+\frac{1}{2})$. Consider now $y$ in a neighborhood of $x$. If $y$ is closer to $e(i-\tfrac{1}{2})$ then $y$ is the midpoint of some pair $(y_1,y_2)$, $y_1\in e(i-\tfrac{1}{2})$ and 
$y_2\in e(i+n+\tfrac{1}{2})$ and is also 
the midpoint of another pair $(y_1',y_2')$, $y_1'\in e(i-\tfrac{1}{2})$ and $y_2'\in e(i+n-\tfrac{3}{2})$. On the other hand, if $y$ is closer to $e(i+n-\tfrac{1}{2})$, there are no pairs $(y_1,y_2)$ in the neighborhood
of $(x_1,x_2)$ or $(x_1',x_2')$ with $y$ as midpoint. We conclude that when we traverse a segment of the area evolute from the smaller to the bigger side, $N(x)$ is increased by $2$. 

\begin{figure}[htb]
 \centering
 \includegraphics[width=1.00\linewidth]{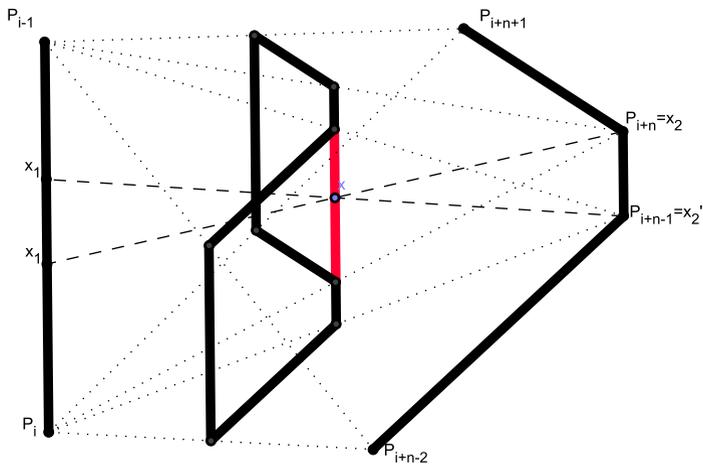}
 \caption{Case 4. A fold at the area evolute.}
\label{MidPoint2}
\end{figure}

\end{enumerate}
\end{proof}

\subsection{Cusps of the area evolute and central symmetry set}

In this section, we make the simplifying assumption that $D(i-\tfrac{1}{2})\neq D(i+\tfrac{1}{2})$, for every $1\leq i\leq n$. Define $\lambda(i+\tfrac{1}{2})$, $1\leq i\leq 2n$,  by 
$$
D(i+\tfrac{1}{2})=P_i+\lambda(i+\tfrac{1}{2})\left( P_{i+n}-P_i\right),
$$
where $D(i+n+\tfrac{1}{2})=D(i+\tfrac{1}{2})$. Observe that
$$
e(i+n+\tfrac{1}{2})=\frac{\lambda(i+\tfrac{1}{2})-1}{\lambda(i+\tfrac{1}{2})} \ e(i+\tfrac{1}{2}), 
$$
for any $1\leq i\leq 2n$ and that $\lambda(i+n+\tfrac{1}{2})=1-\lambda(i+\tfrac{1}{2})$.

We say that a vertex $D(i_0+\tfrac{1}{2})$ of the Central Symmetry Set is a cusp if $\lambda(i_0+\tfrac{1}{2})$ is a local extremum of the cyclic sequence $\lambda(i+\tfrac{1}{2})$. For the Area Evolute, we say that a vertex $M_{i_0}$ is a cusp  if $\lambda(i_0-\tfrac{1}{2})\cdot \lambda(i_0+\tfrac{1}{2})<0$. The geometrical meaning of this definitions will be clear in next section.

\begin{proposition}\label{prop:cuspsCSS}
The number  of cusps of the Area Evolute is odd and $\geq 3$. The number of cusps of the Central Symmetry Set is odd and not smaller than the number of cusps of the Area Evolute. 
\end{proposition}

\begin{proof}
Since $\lambda(i+n+\tfrac{1}{2})=1-\lambda(i+\tfrac{1}{2})$, there are an odd number of local extrema and zero crossings in the sequence $\lambda(i+\tfrac{1}{2})$, $1\leq i\leq n$. Thus the numbers of cusps of the Central Symmetry Set and the Area Evolute are odd. Moreover, between two zero crossings there exists at least one local extremum.
Thus the number of cusps of the Central Symmetry Set is bigger than or equal to the number of cusps of the Area Evolute. 

Assume now by contradiction that the sequence $\lambda(i+\tfrac{1}{2})$, $1\leq i\leq 2n$, has only one zero crossing. Then it would necessarily have $n$ consecutive values bigger than $1/2$ followed by $n$ consecutive values smaller than $1/2$. Thus the sides of the polygon would be of the form 
$w_1,...,w_n$, $-\alpha_1w_1,...,-\alpha_nw_n$,
with $\alpha_i>1$, $1\leq i\leq n$, and 
$$
\sum_{i=1}^nw_i=\sum_{i=1}^n\alpha_iw_i. 
$$
Taking the determinant product with $w_1$ we get
$$
\sum_{i=2}^n[w_1,w_i]=\sum_{i=2}^n\alpha_i[w_1,w_i],
$$
which is a contradiction since $[w_1,w_i]>0$, for $2\leq i\leq n$. 
\end{proof}

\subsection{ Symmetric and non-symmetric CPOS equal-area polygons}

A polygon is called equal-area if 
$$
\left[     P_{i+1}-P_i,     P_{i}-P_{i-1}    \right]=\left[     P_{j+1}-P_j,     P_{j}-P_{j-1}    \right]
$$
for any $i,j$ (\cite{Craizer12}). It is easy to verify that any symmetric CPOS is equal-area,  
but, as we shall see below, the reciprocal is true only for even $n$. 

For $n$ odd, consider vectors $w_1, w_2,...,w_n$ satisfying $[w_i,w_{i+1}]<0$ and $\sum_{i=1}^nw_i=0$. For $\alpha>0$, 
consider a $2n$-gon ${\mathcal P}$ whose sides are $e_i=w_i,\ e_{n+i}=-\alpha w_i$, for $1\leq i\leq n$ odd, and $e_i=-\alpha w_i,\ e_{n+i}=w_i$, for $1\leq i\leq n$ even. Then the polygon ${\mathcal P}$
is  CPOS and equal-area. 

\begin{proposition}
Any CPOS equal-area polygon is symmetric or obtained by the above construction. 
\end{proposition}
\begin{proof}
Let $e_1,e_2...,e_n, -\alpha_1e_1,...,-\alpha_ne_n$ denote the sides of an equal area polygon ${\mathcal P}$.  Then
$$
[e_1,e_2]=[e_2,e_3]=...=-\alpha_1[e_n,e_1]=\alpha_1\alpha_2[e_1,e_2]=...=-\alpha_n[e_n,e_1].
$$
Thus $\alpha_1=\alpha_n$ and $\alpha_i\alpha_{i+1}=1$. If $\alpha_i=1$ for some $i$, then $\alpha_i=1$ for any $i$ and the polygon is symmetric. If $\alpha_i\neq 1$, then necessarily
$n$ is odd and $\alpha_{i+1}=\alpha_i^{-1}$. Thus the polygon is obtained by the above construction.
\end{proof}

We shall refer to the above constructed polygons as non-symmetric equal-area CPOS polygons (see figure \ref{CPOSEA} and Applets). 

\begin{figure}[htb]
 \centering
 \includegraphics[width=1.0\linewidth]{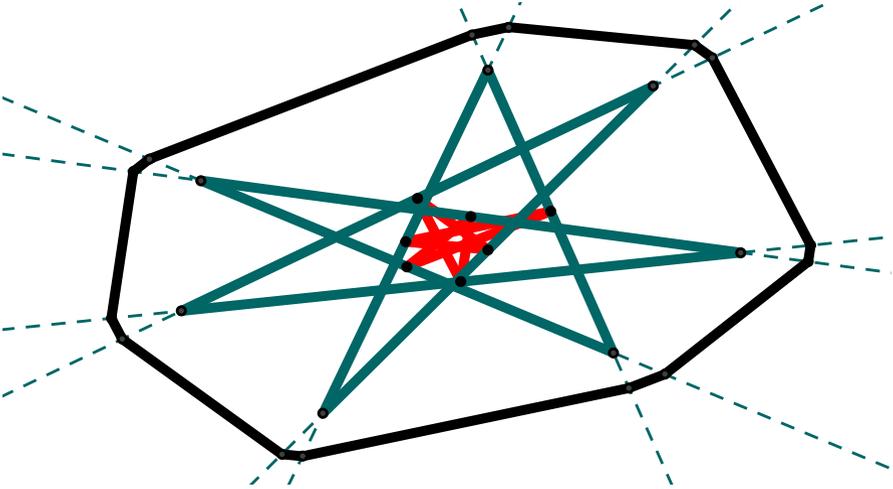}
 \caption{A CPOS $14$-gon non-symmetric equal-area. }
\label{CPOSEA}
\end{figure}

\begin{proposition}
Consider a non-symmetric equal-area CPOS polygon ${\mathcal P}$. Then $M_i$ is the midpoint of $D(i-\tfrac{1}{2})D(i+\tfrac{1}{2})$. As a consequence,
the Area Evolute and the Central Symmetry Set have $n$ cusps, i.e., all vertices are cusps. 
\end{proposition}

\begin{proof}
Write $P_{i+1}-P_i=-\alpha \left( P_{i+n+1}-P_{i+n} \right)$ for some $\alpha>0$. Then
$P_{i+n}-P_{i+n-1}=-\alpha\left( P_{i}-P_{i-1}\right)$. Taking
$\lambda=\frac{\alpha-1}{2(1+\alpha)}$, straightforward calculations show that
$$
D(i\pm\tfrac{1}{2})=M_i\pm\lambda(P_{i+n}-P_i),
$$
which proves that $M_i$ is the mid-point of $D(i-\tfrac{1}{2})D(i+\tfrac{1}{2})$. From this, one easily concludes
that each point of the Central Symmetry Set is a cusp. 
\end{proof}

\section{Equidistants}\label{sec:equidistants}

An {\it equidistant} at level $\lambda\in\R$ is the polygon ${\mathcal P}_{\lambda}$ whose vertices are
$$
P_i(\lambda)= P_i+\lambda (P_{n+i}-P_i).
$$
For $\lambda=\frac{1}{2}$, the equidistant is exactly the area evolute (see figure \ref{Equidistants} and Applets). 
The edges of ${\mathcal P}_{\lambda}$ are $e(i+\tfrac{1}{2})(\lambda)=P_{i+1}(\lambda)-P_{i}(\lambda)$. Observe that
$$
e(i+\tfrac{1}{2})(\lambda)=(1-\lambda)e(i+\tfrac{1}{2})+\lambda e(i+n+\tfrac{1}{2}).
$$

\subsection{ Cusps of the equidistants and the Central Symmetry Set}

Denote $f_i(\lambda)=[e(i-\tfrac{1}{2})(\lambda), e(i+\tfrac{1}{2})(\lambda)]$. Observe that $f_i(0)>0$ and $f_i$ changes sign only when
$e(i-\tfrac{1}{2})(\lambda)$ or $e(i+\tfrac{1}{2})(\lambda)$ vanishes. Thus $f_i(\lambda)<0$ if an only if $e(i-\tfrac{1}{2})(\lambda)$ and $e(i+\tfrac{1}{2})(\lambda)$
are in the same half-plane determined by $d_i$.
We say that a vertex $P_i(\lambda)$ is a {\it cusp} of the equidistant ${\mathcal P}_{\lambda}$ if $f_i(\lambda)<0$. 

\begin{proposition}
The set formed by the cusps of the $\lambda$-equidistants coincides with the Central Symmetry Set.
\end{proposition}

\begin{proof}
We write
$$
e(i\pm\tfrac{1}{2})(\lambda)=(1-\lambda)e(i\pm\tfrac{1}{2})+\lambda e(i+n\pm\tfrac{1}{2}).
$$
Thus $f_i(\lambda)$ is a quadratic function that vanishes at $\lambda(i-\tfrac{1}{2})$ and $\lambda(i+\tfrac{1}{2})$ corresponding 
to the intersections of $d_i$ with $d_{i-1}$ and $d_{i+1}$, respectively. Thus $f_i(\lambda)<0$ only for $\lambda$ between
$\lambda(i-\tfrac{1}{2})$ and $\lambda(i+\tfrac{1}{2})$, which correspond to points of the Central Symmetry Set.
\end{proof}

\begin{figure}[htb]
 \centering
 \includegraphics[width=1.0\linewidth]{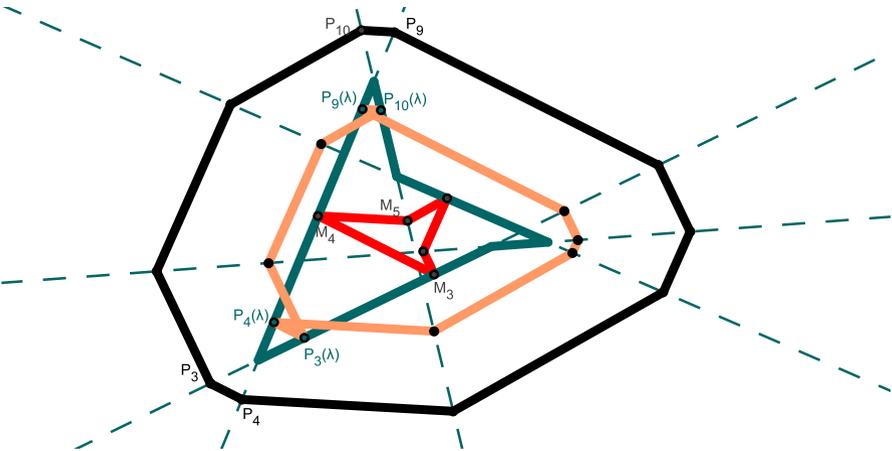}
 \caption{An equidistant with vertices $P_i(\lambda )$ with $\lambda=0.2$. The labeled vertices $P_i(\lambda )$ are cusps for this equidistant. When $\lambda = 1/2$, note that $M_1$, $M_3$ and $M_4$, which are on the Central Symmetry Set, are cusps, but $M_2$ and $M_5$ are not. }
\label{Equidistants}
\end{figure}

\subsection{Equidistant symmetry set}

The self intersection of the equidistants form a set that we call {\it equidistant symmetry set} (see figure \ref{EqSymSet} and Applets).

Given non-parallel sides $e(i+\tfrac{1}{2}),\ e(j+\tfrac{1}{2})$, denote by $P(i+\tfrac{1}{2},j+\tfrac{1}{2})$ the intersection of their support lines. Denote also by $l(i+\tfrac{1}{2},j+\tfrac{1}{2})$
the line passing through $P(i+\tfrac{1}{2},j+\tfrac{1}{2})$ and $P(i+n+\tfrac{1}{2},j+n+\tfrac{1}{2})$. Observe that the support line of an edge of the Equidistant Symmetry Set associated with a pair of sides 
$(i+\tfrac{1}{2},j+\tfrac{1}{2})$ is exactly $l(i+\tfrac{1}{2},j+\tfrac{1}{2})$. 

Consider two edges $(i-\tfrac{1}{2},j+\tfrac{1}{2})$ and $(i+\tfrac{1}{2},j+\tfrac{1}{2})$ of the Equidistant Symmetry Set with a common vertex $(i,j+\tfrac{1}{2})$. We say that the vertex is a {\it cusp}
if both edges are at the same side of $d_i$.

\begin{lemma}
A vertex of the Equidistant Symmetry Set which is not an endpoint is a cusp if and only if it belongs to the Central Symmetry Set. 
\end{lemma}
\begin{proof}
The support lines of $e(i-\tfrac{1}{2})(\lambda)$ and $e(i+\tfrac{1}{2})(\lambda)$ determine a planar region which does not contain any other support
line of $e(j+\tfrac{1}{2})(\lambda)$, $j\neq i-1,i$. We shall denote this region by $R_i(\lambda)$.

Consider a vertex $(i,j+\tfrac{1}{2})$ at level $\lambda_0$. Assume the $P_{i}(\lambda_0)$ is at the Central Symmetry Set. Then the $\lambda_0$-equidistant has a cusp at this point. Observe that
the support line of $e(j+\tfrac{1}{2})(\lambda)$ is not contained in $R_i(\lambda)$. 
Thus, for $\lambda$ close to $\lambda_0$,
the edge $e(j+\tfrac{1}{2})(\lambda)$ intersects $e(i+\tfrac{1}{2})(\lambda)$ and $e(i-\tfrac{1}{2})(\lambda)$ only for $\lambda>\lambda_0$ or $\lambda<\lambda_0$. 
For definiteness, we shall assume that these intersections occur for $\lambda>\lambda_0$. Thus, for $\lambda>\lambda_0$ we have two different intersections of the equidistants and thus
the Equidistant Symmetry Set will be on one side of $d_i$, which means a cusp.

Conversely, if the vertex $(i,j+\tfrac{1}{2})$ is not on the Central Symmetry Set, then the $\lambda_0$-equidistant have a regular point. For $\lambda$ close to $\lambda_0$,
the edge $e(j+\tfrac{1}{2})(\lambda)$ intersects $e(i+\tfrac{1}{2})(\lambda)$ and $e(i-\tfrac{1}{2})(\lambda)$ once. Thus the corresponding edges of the Central Symmetry Set cross the line $d_i$, and thus 
is a regular point of the Equidistant Symmetry Set.
\end{proof}

\begin{proposition}\label{BranchESS}
Every branch of the Equidistant Symmetry Set can be continued until it reaches a cusp of the Central Symmetry Set or a cusp of the Area Evolute (see figure \ref{EqSymSet}).  
\end{proposition}

\begin{proof}
Consider an edge $e(i+\tfrac{1}{2},j+\tfrac{1}{2})$ of the Equidistant Symmetry Set. This edge will end at the great diagonal  $d_i,d_{i+1},d_j$ or $d_{j+1}$. For definiteness, assume it is $d_{j+1}$. 
Then we can continue the Equidistant Symmetry Set by the edge $e(i+\tfrac{1}{2},j+1+\tfrac{1}{2})$. This is always possible except in two cases: 
(1) $|j+1-i|=1$ and we are at a cusp of the Central Symmetry Set or (2) $j+1=i+n$ and we are at a cusp of the Area Evolute. 
\end{proof}

\begin{figure}[htb]
\centering \fsep \subfigure[ A branch of the Equidistant Symmetry Set with both endpoints at the cusps of the Central Symmetry Set.] {
\includegraphics[width=.45
\linewidth,clip =false]{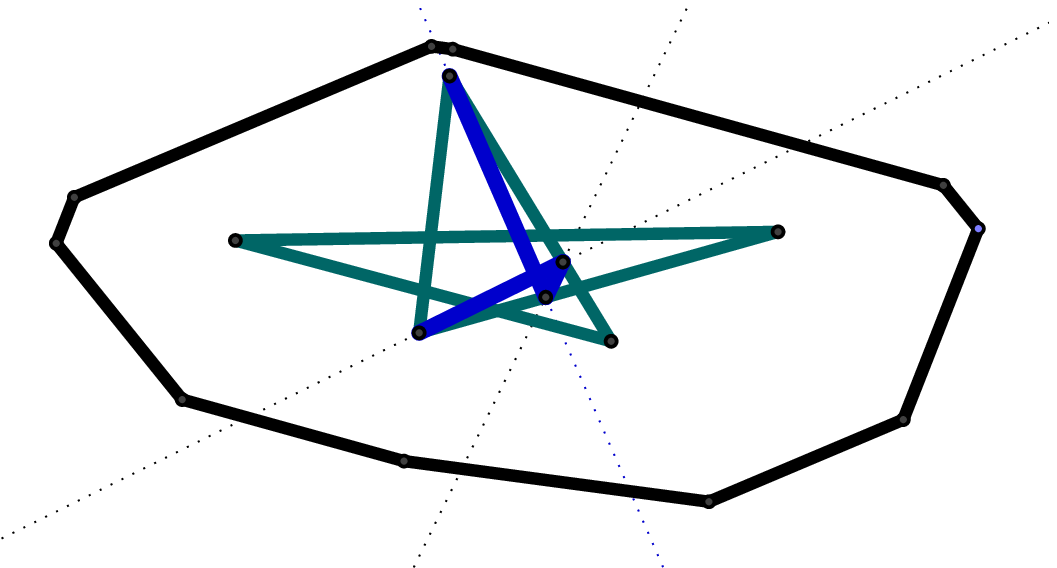}} \fsep\subfigure[
A branch of the Equidistant Symmetry Set with both endpoints at the cusps of the Area Evolute. ] {
\includegraphics[width=.45\linewidth,clip
=false]{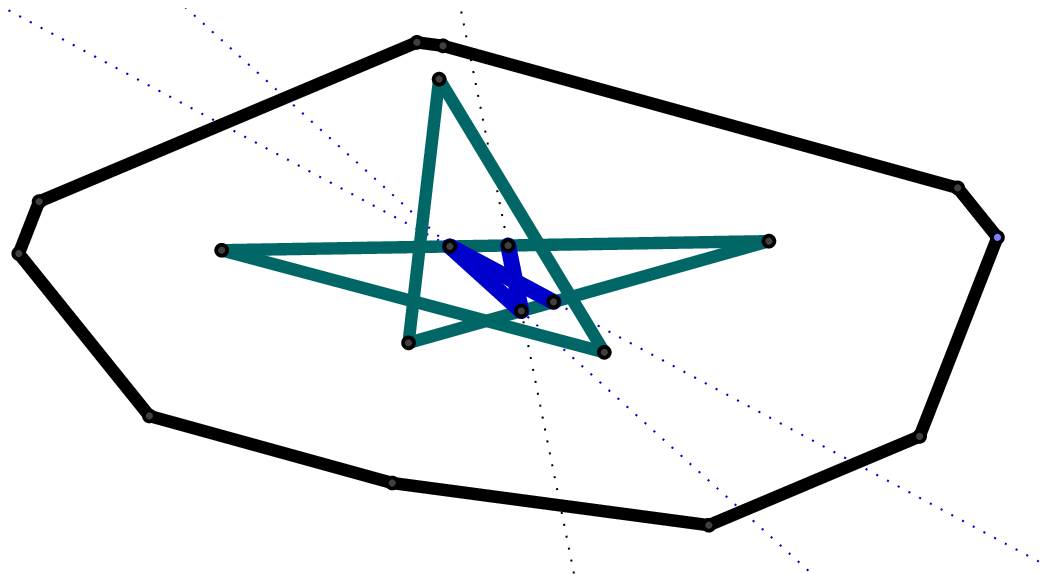}}\fsep
\caption{Branches of the Equidistant Symmetry Set.}
\label{EqSymSet}
\end{figure}

\section{The parallel-diagonal transform}\label{sec:central}

To a CPOS polygons ${\mathcal P}$, we can associate another CPOS polygon ${\mathcal Q}$ 
whose vertices are at the mid-parallel lines and whose sides are parallel to the great diagonals  of ${\mathcal P}$ (see figure \ref{PD} and Applets). 
In this section we prove the existence of such a CPOS polygon ${\mathcal Q}$, called the { \it parallel-diagonal transform} of ${\mathcal P}$.

\begin{figure}[htb]
 \centering
 \includegraphics[width=1.0\linewidth]{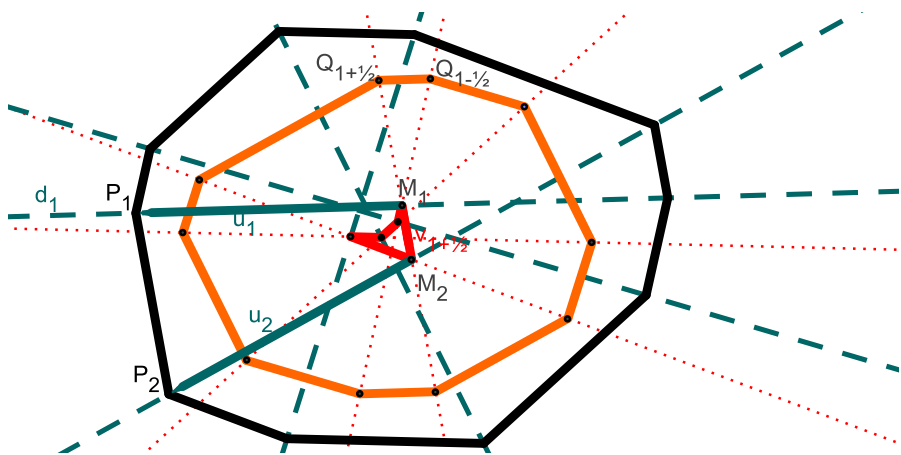}
 \caption{The parallel-diagonal transform. }
\label{PD}
\end{figure}

Denote $v(i+\tfrac{1}{2})=M_{i+1}-M_i$ and  $u_i=P_i-M_i$, $1\leq i\leq n$. Define also $v(i+n+\tfrac{1}{2})=v(i+\tfrac{1}{2})$ and 
$u_{i+n}=-u_i$. Observe that 
\begin{equation}\label{relationuv}
\left[v(i+\tfrac{1}{2}),u_i-u_{i+1}\right]=0 
\end{equation}
and 
\begin{equation}\label{sumvequal0}
\sum_{k=1}^nv(k+\tfrac{1}{2})=0.
\end{equation}

\begin{lemma}\label{lemma:areas}
Denote by $A^1$ and $A^2$ the areas of the regions of the polygon ${\mathcal P}$ bounded by the great diagonal $d_1$.  Then
$$
A^2-A^1=2\sum_{j=1}^n [v(j+\tfrac{1}{2}), u_j]
$$
\end{lemma}
\begin{proof}
Observe that
$$
2A(M_1P_jP_{j+1})=\left[  \sum_{k=1}^{j-1}v(k+\tfrac{1}{2})+u_j,   \sum_{k=1}^{j}v(k+\tfrac{1}{2})+u_{j+1}  \right]
$$
$$
2A(M_1P_{j+n}P_{j+n+1})=\left[  \sum_{k=1}^{j-1}v(k+\tfrac{1}{2})-u_j,   \sum_{k=1}^{j}v(k+\tfrac{1}{2})-u_{j+1}  \right]
$$
and thus the difference between these areas is 
$$
\Delta A(j)=\left[  \sum_{k=1}^{j}v(k+\tfrac{1}{2}),   u_{j}  \right]- \left[  \sum_{k=1}^{j-1}v(k+\tfrac{1}{2}),   u_{j+1}  \right].
$$
So $A^2-A^1$ is given by 
$$
\sum_{j=1}^n\Delta A(j)=\sum_{k=1}^{n-1}\left[ v(k+\tfrac{1}{2}),  \sum_{j=k+1}^n u_j-u_{j+1} \right]+\sum_{j=1}^n\left[ v(j+\tfrac{1}{2}), u_j \right]
$$
$$
= \sum_{k=1}^{n-1}\left[ v(k+\tfrac{1}{2}),   u_{k+1}+u_{1} \right]+\sum_{j=1}^n\left[ v(j+\tfrac{1}{2}), u_j \right]=2\sum_{j=1}^n\left[ v(j+\tfrac{1}{2}), u_j \right],
$$
where we have used \eqref{relationuv} and \eqref{sumvequal0}. 
\end{proof}

Denote by $N(i+\tfrac{1}{2})$ the point of the mid-parallel $m(i+\tfrac{1}{2})$ satisfying 
\begin{equation*}
[N(i+\tfrac{1}{2})-M_i, P_{i+n}-P_i]= \frac{1}{2}\left( A_i^2-A_i^1\right),
\end{equation*}
where $A_i^1$ and $A_i^2$ denote the areas of the regions of the polygon ${\mathcal P}$ bounded by the great diagonal $d_i$. Denote by ${\mathcal N}$ the polygon 
with vertices $N(i+\tfrac{1}{2})$. 

\begin{proposition}
Given a CPOS polygon ${\mathcal P}$, we can associate a CPOS polygon ${\mathcal Q}$ with vertices at the mid-parallel lines  and 
sides parallel to the great diagonals of ${\mathcal P}$. The Area Evolute of ${\mathcal Q}$ is exactly ${\mathcal N}$. 
\end{proposition}

\begin{proof}
Start with a point 
$$
Q(1-\tfrac{1}{2})=M_n+\mu v(1-\tfrac{1}{2})\in m(1-\tfrac{1}{2})
$$
and then follow the parallel to the great diagonal $d_1$ until
it reaches the line $m(1+\tfrac{1}{2})$ at $Q(1+\tfrac{1}{2})$. Then take the parallel to $d_2$ until intersect  the line $m(2+\tfrac{1}{2})$ at $Q(2+\tfrac{1}{2})$. Follow this algorithm and stop after $2n$ steps,
when the polygon reaches the line $m(1-\tfrac{1}{2})$ at $Q(2n+\tfrac{1}{2})$. 
We must prove that $Q(2n+\tfrac{1}{2})=Q(1-\tfrac{1}{2}) $. 

Define $\mu(i+\tfrac{1}{2})$,  $1\leq i\leq 2n$, by the relation
$$
Q(i+\tfrac{1}{2})=M_i+\mu(i+\tfrac{1}{2})v(i+\tfrac{1}{2}).
$$
We must then prove that $\mu(2n+\tfrac{1}{2})=\mu(1-\tfrac{1}{2})$, where $\mu(1-\frac{1}{2})=\mu$. Denoting
$q_i=Q(i+\tfrac{1}{2})-Q(i-\tfrac{1}{2})$, $1\leq i \leq 2n$, we have 
\begin{equation*}
q_i=\left(1-\mu(i-\tfrac{1}{2})\right)v(i-\tfrac{1}{2})+\mu(i+\tfrac{1}{2})v(i+\tfrac{1}{2}).
\end{equation*}
Since $q_i$ must be parallel to $u_i$ we obtain
\begin{equation}\label{relationmu}
\mu(i+\tfrac{1}{2})[v(i+\tfrac{1}{2}),u_i]+\left(1-\mu(i-\tfrac{1}{2})\right)[v(i-\tfrac{1}{2}),u_i]=0. 
\end{equation}

Denoting $\beta(i+\tfrac{1}{2})=\mu(i+\tfrac{1}{2})-\mu(i+n+\tfrac{1}{2})$, we get
$$
\beta(i+\tfrac{1}{2})[v(i+\tfrac{1}{2}),u_i]=\beta(i-\tfrac{1}{2})[v(i-\tfrac{1}{2}),u_{i-1}],
$$
where we have used \eqref{relationuv}. We conclude that $\beta(1-\tfrac{1}{2})=-\beta(n+\tfrac{1}{2})$,
thus proving that  $\mu(2n+\tfrac{1}{2})=\mu(1-\tfrac{1}{2})$.

Now let
$$
{\tilde N}(i+\tfrac{1}{2})=\frac{1}{2}\left( Q(i+\frac{1}{2})+ Q(i+n+\frac{1}{2}) \right)=M_i+\frac{1}{2} \gamma(i+\frac{1}{2}) v(i+\frac{1}{2}), 
$$ 
where $\gamma(i+\frac{1}{2})= \mu(i+\frac{1}{2})+ \mu(i+n+\frac{1}{2})$ denote the vertices of the Area Evolute of ${\mathcal Q}$. 
We have 
$$
\left[{\tilde N}(i+\tfrac{1}{2})-M_i, P_{i+n}-P_i\right]=  \gamma(i+\frac{1}{2})[ v(i+\tfrac{1}{2}), u_i].
$$
We claim that ${\tilde N}(i+\tfrac{1}{2})=N(i+\tfrac{1}{2})$, for any $0\leq i\leq n-1$. The proof of this claim will now be given for $i=0$, the other cases being similar. 
It follows from \eqref{relationmu} that
$$
\gamma(i+\frac{1}{2})[ v(i+\tfrac{1}{2}), u_i]=\gamma(i-\frac{1}{2})[ v(i-\tfrac{1}{2}), u_i]+2[ v(i-\tfrac{1}{2}), u_i].
$$
Summing from $i=1$ to $i=n$ we obtain
$$
\gamma(\frac{1}{2})[ v(\tfrac{1}{2}), u_1]=\sum_{j=1}^n [v(j+\tfrac{1}{2}), u_j]=\frac{1}{2} \left( A^2-A^1\right),
$$
where the last equality follows from lemma \ref{lemma:areas}. Thus the claim is proved.  

It remains to show that we can choose $\mu$ such that ${\mathcal Q}$ is convex. Take $\mu$ sufficiently large such that the corresponding ${\mathcal Q}$
contains the Area Evolute of ${\mathcal P}$ in its interior. For such $\mu$, one easily verify that ${\mathcal Q}$ has no self-intersections. Since its sides 
are parallel to $u_i$, $1\leq i\leq n$, which are cyclically ordered, we conclude that ${\mathcal Q}$ is convex. 
\end{proof}

\begin{corollary}\label{cor:halfarea}
Suppose that the line $P_iN(i+\tfrac{1}{2})$ intersects the line $P_{n+i}P_{n+i+1}$ at a point $P_i'$ at the segment $P_{n+i}P_{n+i+1}$. Then $N(i+\tfrac{1}{2})$ 
is the mid-point of the chord $P_iP_i'$ which divides the polygon ${\mathcal P}$ in two regions of equal areas. 
\end{corollary}

\section{Rectified area parallels}

Consider a simple closed curve $\gamma$ that bounds a convex region. 
For $x$ inside ${\gamma}$, there may exist more than one chord $c$ with $x$ as midpoint.  Each chord $c$ divide the interior of ${\gamma}$ in two regions, and we shall denote by $A(x,c)$ the smallest area among these regions. For any $0\leq\lambda\leq\frac{A({\gamma})}{2}$, the area parallel of level $\lambda$ is the set of points $x$ for which there exists $c$ with $x$ as midpoint and $A(x,c)=\lambda$. 
Equivalently, the area parallel of level $\lambda$ is the envelope of chords that cut off an area $\lambda$ from the interior of ${\gamma}$. 

\subsection{Area parallels of a CPOS polygon}\label{sec:structureparallels}

For each pair of edges $e(i+\tfrac{1}{2})$ and $e(j+\tfrac{1}{2})$, denote by $M(i+\tfrac{1}{2},j+\tfrac{1}{2})$ the parallelogram 
consisting of the mid-points of the edges, as in section \ref{sec:mid-point}.
The area parallels of the wedge bounded by the support lines of $e(i+\tfrac{1}{2})$ and $e(j+\tfrac{1}{2})$
are hyperbolas, and by intersecting any of these hyperbolas with $M(i+\tfrac{1}{2},j+\tfrac{1}{2})$ we obtain an arc of hyperbola $H(i+\tfrac{1}{2},j+\tfrac{1}{2},\lambda)$
that is part of the area parallel of level $\lambda$ of ${\mathcal P}$.

An area parallel $H(i+\tfrac{1}{2},j+\tfrac{1}{2},\lambda)$ necessarily intersects the boundary of the parallelogram $M(i+\tfrac{1}{2},j+\tfrac{1}{2})$ at two points. 
Next lemma describes the line segments $L(\lambda)=L(i+\tfrac{1}{2},j+\tfrac{1}{2},\lambda)$ connecting these two points.

Let $A,B,C,D$ and $E$ denote the midpoints of $P_{i+1}P_j$, $P_{i}P_j$, $P_{i+1}P_{j+1}$, $P_{i}P_{j+1}$ and $P_{i}P_{i+1}$, respectively. 
Assuming that the area parallel through $B$ intersects the segment $AC$, denote by $F$ this intersection. Then the area parallel through $C$ intersects
$BD$ at a point $G$ (see figure \ref{bilinear}).

\begin{figure}[htb]
 \centering
 \includegraphics[width=1.0\linewidth]{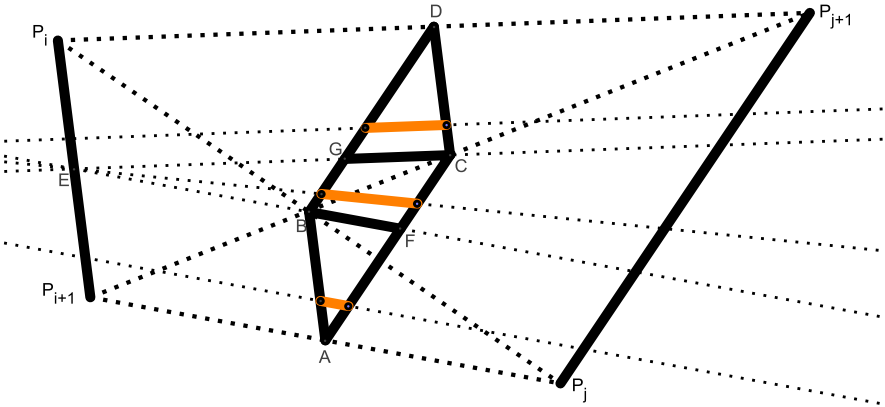}
 \caption{Area parallels inside a parallelogram.}
\label{bilinear}
\end{figure}

\begin{lemma}
If the segment $L(\lambda)$ passes through $BG$, then its support line passes
through $E$. If $L(\lambda)$ passes through $GD$, then it is parallel to $P_{i}P_{j+1}$. If $L(\lambda)$ passes through
$AB$, then it is parallel to $P_{i+1}P_{j}$.
\end{lemma}
\begin{proof}
We may assume that the segment $P_iP_{i+1}$ is contained in the $y$-axis and the segment $P_jP_{j+1}$ is contained in the $x$-axis. In this case the area 
at a point $(x,y)$ is $xy$. Then straightforward calculations prove the lemma.
\end{proof}

\subsection{Rectified area parallels and its cusps}

We shall call {\it rectified area parallel} of level $\lambda$ 
the union of all segments $L(i+\tfrac{1}{2},j+\tfrac{1}{2},\lambda)$. Thus the rectified area parallels are closed polygons (see figure \ref{AreaParallel} and Applets). 

\begin{figure}[htb]
 \centering
 \includegraphics[width=1.0\linewidth]{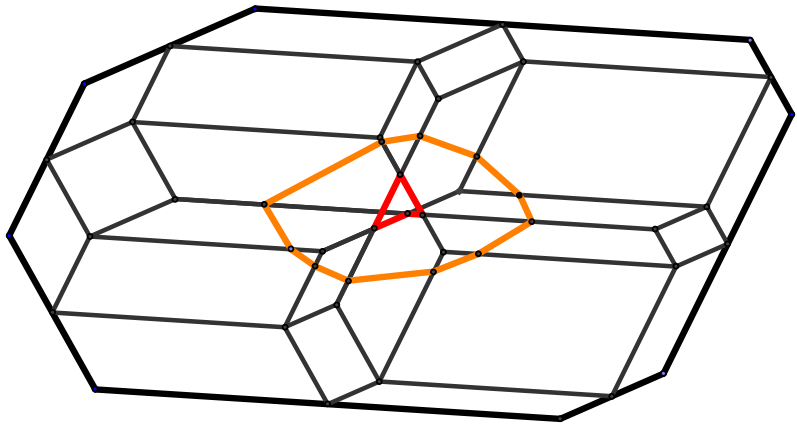}
 \caption{A rectified area parallel of a parallel opposite sides octogon. }
\label{AreaParallel}
\end{figure}

A rectified area parallel has its vertices at the boundaries of the parallelograms $M(i-\tfrac{1}{2},j-\tfrac{1}{2})$. Consider a vertex 
$x$ at the boundary $(i-\tfrac{1}{2},j)$. The rectified area parallel passing through $x$ at the parallelogram $M(i-\tfrac{1}{2},j+\tfrac{1}{2})$ will continue the initial polygonal line 
unless, $j=i-1+n$, i.e.,  $e(i-\tfrac{1}{2})$ is parallel to $e(j+\tfrac{1}{2})$.
In the latter case, the continuation of the rectified area parallel may occur at the same side of $m(i-\frac{1}{2})$ or on the other side. We call $x$ a cusp when the continuation
of the area parallel occurs at the same side of $m(i-\frac{1}{2})$. We have the following proposition:

\begin{proposition}\label{prop:cuspsareaparallels}
A vertex of a rectified area parallel is a cusp if and only if it belongs to the area evolute. 
\end{proposition}
\begin{proof}
There are two cases to consider: if $x$ is not in the area evolute, then the rectified area parallel will continue in the parallelogram $M(i+\tfrac{1}{2},i+n-\tfrac{1}{2})$ and thus $x$ is not a cusp (see figure \ref{MidPoint1}). 
If $x$ is in the area evolute, the rectified area parallel will continue in the parallelogram $M(i-\tfrac{1}{2},i+n+\tfrac{1}{2})$ (see figure \ref{MidPoint2}). Close to $x$, the area parallel is contained in the same 
side of $m(i-\tfrac{1}{2})$ and thus $x$ is a cusp point. 
\end{proof}

\subsection{Rectified area symmetry set}

The {\it rectified area symmetry set} is the locus of self-intersections of the rectified area parallels. It is not easy to describe this set for a general CPOS polygon, so we shall make a simplifying assumption. Denote by ${\mathcal Q}_{\mu}$ the family of equidistants of the parallel-diagonal transform ${\mathcal Q}$ of ${\mathcal P}$. 
We say that ${\mathcal P}$ is {\it almost symmetric} if there exists $\mu_0$ such that ${\mathcal Q}_{\mu_0}$ contains the Area Evolute in its interior
and every mid-point of $1$-diagonals are outside it (see figure \ref{AlmostSym}).

\begin{figure}[htb]
 \centering
 \includegraphics[width=1.0\linewidth]{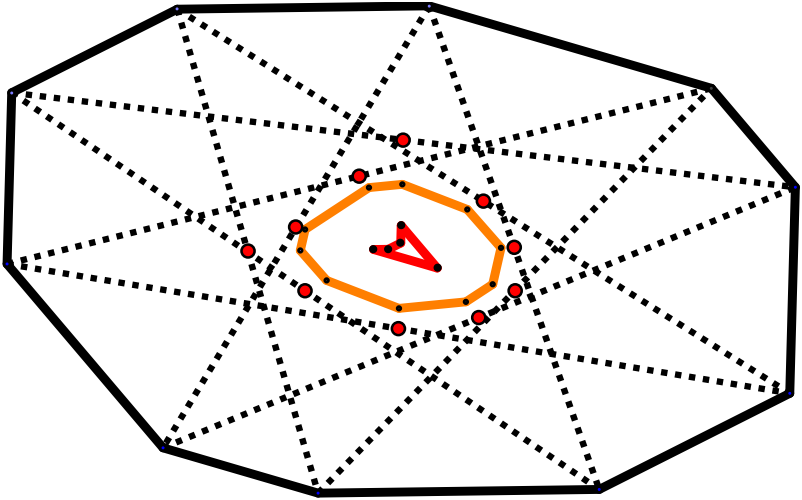}
 \caption{Area evolute, an equidistant of ${\mathcal Q}$ and the $10$ midpoints of $1$-diagonals (circles). The polygon ${\mathcal P}$ is almost symmetric. }
\label{AlmostSym}
\end{figure}

Under the almost symmetry hypothesis, the polygons ${\mathcal Q}_{\mu}$, $\mu\geq\mu_0$, are rectified area parallels of ${\mathcal P}$. In fact, it follows from \ref{sec:structureparallels} that each edge of a rectified area parallel is
parallel to the corresponding great diagonal. In particular, using corollary \ref{cor:halfarea}, we have that the rectified area parallel of level $\frac{A({\mathcal P})}{2}$ is exactly the Area Evolute of ${\mathcal Q}$. 

\begin{proposition}
Under the almost symmetry hypothesis, the Rectified Area Symmetry Set of ${\mathcal P}$ coincides with the Equidistant Symmetry Set of ${\mathcal Q}$.
\end{proposition}

\begin{proof}
We first observe that, as a consequence of proposition \ref{prop:cuspsareaparallels}, if a rectified area parallel contains the Area Evolute in its interior, then it has no self-intersections. 
Thus, by the almost symmetry hypothesis, the self-intersections of rectified area parallels occur only
for those levels corresponding to ${\mathcal Q}_{\mu}$, $\mu\geq\mu_0$. We conclude that the Rectified Area Symmetry Set of ${\mathcal P}$ coincides with the Equidistant Symmetry Set of ${\mathcal Q}$.
\end{proof}



\begin{thebibliography}{}

\bibitem{Bobenko08} A.I.Bobenko and Y.B.Suris, Discrete Differential Geometry. Graduate Studies in Mathematics, v.98, AMS (2008).

\bibitem{Craizer08} M. Craizer, R.C.Teixeira and M.A.H.B.da Silva, Area Distances of Convex Plane Curves and Improper Affine Spheres, SIAM Journal on Mathematical Imaging, 1(3), p.209-227, 2008.

\bibitem{Craizer12} M. Craizer, R.C.Teixeira and M.A.H.B.da Silva, Affine properties of convex equal-area polygons. Discrete and Computational Geometry, 48(3), 580-595, (2012).


\bibitem{Giblin08} P.J.Giblin, Affinely invariant symmetry sets. Geometry and Topology of Caustics, Banach Center Publications, vol.82 (2008), p.71-84.

\bibitem{Holtom99} P.J.Giblin and P.Holtom. The Centre Symmetry Set. Banach Center Publications, vol.50 (1999), p.91-105.

\bibitem{Holtom01} P.A.Holtom, Affine-Invariant Symmetry Sets, PH.D. thesis, University of Liverpool, 2001. 

\bibitem{Janeczko96} S.Janeczko, Bifurcations of the center of symmetry. Geometriae Dedicata 60, 9-16, (1996).

\bibitem{Meyer91} M.Meyer and S.Reisner, A Geometric Property of the Boundary of Symmetric Convex Bodies and Convexity of Flotation Surfaces. Geometriae Dedicata 37 (1991) 327-337. 

\bibitem{Niethammer04} M.Niethammer, S.Betelu, G.Sapiro, A.Tannenbaum, P.J.Giblin, Area-based Medial Axis of Planar Curves. Int. J. Computer Vision 60 (2004), 203-224. 

\bibitem{Rios11}
W.Domitrz and  P.M.Rios, Singularities of equidistants and global symmetry sets of Lagrangian submanifolds, arXiv: 1007.1939.

\bibitem{GeoGebra} http://www.geogebra.org/cms.

\bibitem{Applets} Personal home page of the second author http://www.professores.uff.br/ralph.








 
\end{thebibliography}
\end{document}